\documentclass[12pt]{article}
\usepackage[centertags]{amsmath}
\usepackage{amsfonts, amsthm, amssymb}
\usepackage{graphicx}
\usepackage{float}
\usepackage{dsfont}
\usepackage{cite}
\usepackage{epstopdf}
\usepackage{subfigure}
\usepackage{hyperref}
\usepackage{yhmath}
\newtheorem{thm}{Theorem}[section]
\newtheorem{cor}[thm]{Corollary}
\newtheorem{lem}[thm]{Lemma}

\theoremstyle{definition}

\theoremstyle{remark}

\numberwithin{equation}{section}

\title{\bf The orthogonal connectedness of polyhedral surfaces}
\author{Julia Q. Du$^{1, 2, 3, 4}$, Xuemei He$^{1}$,  Xiaotian Song$^{1}$,\\  Daniela Stiller$^{5}$, Liping Yuan\footnote{Corresponding author. Email: lpyuan@hebtu.edu.cn} $^{1, 2, 3, 4}$\ \
and Tudor Zamfirescu$^{1, 3, 6, 7}$
}

\date{}

\begin{document}

\vskip -8cm \maketitle \thispagestyle{plain}

{\small 1. School of Mathematical Sciences,
Hebei Normal University,
050024 Shijiazhuang, P.R. China.

2. Hebei Key Laboratory of Computational Mathematics and Applications, 050024 Shijiazhuang,  P.R. China.

3. Hebei International Joint Research Center for Mathematics and Interdisciplinary Science,
050024 Shijiazhuang, P.R. China.

4. Hebei Research Center of the Basic Discipline Pure Mathematics,
050024 Shijiazhuang, P.R. China.

5. Integrierte Gesamtschule, Friedrich Wilhelm Raiffeisen, Hamm/Sieg, Germany

6. Fachbereich Mathematik, Technische Universit\"at Dortmund,
44221 Dortmund, Germany

7. Roumanian Academy,  014700 Bucharest, Roumania
}
\begin{abstract}
Using the orthogonal connectedness, we introduce the notion of
orthogonal decomposability of convex polytopes and
study it in the case of Platonic and
Archimedean solids.
While doing so, we also encounter polytopes
which are not orthogonally decomposable.


\end{abstract}

\section{Introduction}

In some application areas of computational geometry, for example concerning very large-scale integrated (VLSI) circuits, and
 digital image processing, the  lines in the Euclidean plane parallel to
 the $x$- or $y$-axis are of major interest.  This is so, because orthogonal polygons, defined as connected
 unions of finitely many planar boxes whose edges are parallel to the coordinate axes, are
often used as building blocks for VLSI layout and wire routing \cite{UAR02},
and also used in image
 processing \cite{SCS16} to describe images on rectangular lattices.

In 1984, Ottmann,  Soisalon-Soininen and  Wood \cite{ottm} introduced  orthogonal paths and the orthogonal connectedness.
A polygonal path $P\subset \mathds{R}^{3} $ is \emph{orthogonal}, if every edge of $P$ is parallel to one of the coordinate axes.
Let $S\subset \mathds{R}^{3}$ be a set. If, for any two distinct points $p,q\in S$, there is an orthogonal path $P$ joining $p$ and $q$ such that $ P\subset S$, then we say that  $S$ is \emph{orthogonally connected}.

Decomposition of polytopes into solids is a
lively topic in computational geometry.
The main purpose behind decomposition operations is to simplify
a problem for complex objects into a number of subproblems dealing with simple objects.
Since convex shapes are
considered useful for representation, manipulation and
rendering, a decomposition into convex solids is sought,
for example, convex decompositions of polytopes \cite{BD92,ZTS02,CZ09},
the triangulation of polytopes \cite{G60,Y05,Y10,IZ04,IZ07}.

There are some polyhedral surfaces in $ \mathds{R}^{3}$ which are orthogonally connected, such as
the boundary of a parallelotope with  all edges parallel to the coordinate axes, or the boundary of the connected union of
a \mbox{finite} family of such parallelotopes.
However, not all polyhedral surfaces in $\mathds{R}^{3}$
are orthogonally connected.
For example, the boundary of the regular octahedron is
not orthogonally connected, for any position in $\mathds{R}^3$.
In this paper, we discuss the orthogonal connectedness of
polyhedral surfaces in $\mathds{R}^{3}$ and
give partitions of polytopes into polytopes whose boundaries are orthogonally connected for Platonic and
Archimedean solids.

Our goal is to deepen the general knowledge about orthogonal connectedness. The rest of this paper is organized as follows.
We introduce some definitions
and notation in Section \ref{Sect-def}.
In Section \ref{Sect-ocp}, we investigate polytopes
whose boundaries are orthogonally connected.
Section \ref{Sect-odps} introduces the
orthogonal decomposability and presents
decompositions of the regular octahedron
and of the regular tetrahedron into polytopes
whose boundaries are orthogonally connected.
In Section \ref{Sect-odas},
we study the orthogonal decomposability of  Archimedean solids.
Polytopes which are not orthogonally decomposable
are investigated in Section \ref{Sect-ond}.

\section{Definitions and notation}\label{Sect-def}

For a  set  $M\subset \mathds{R}^{3}$, we denote by $\mathrm{  conv}M$ its convex hull,  by $\overline{M}$ its affine hull, and  by  $\mathrm {cl}M, \mathrm { bd}M $ its closure and relative boundary, which means in the topology of $\overline{M}$.

Put $x_{1}x_{2}\ldots x_{n}=\mathrm {conv}\{x_{1},x_{2},\ldots ,x_{n}\}$, for $x_{1},\ldots ,x_{n}\in \mathds{R}^{3}$.
Thus,   $xy$  denotes the line-segment from $x$ to $y$, and
$\overline{xy}$ the line through $x$ and $y$.
Let $[x_{1},x_{2},\ldots, x_{n}]$ be the union of line-segments $x_ix_{i+1}$ $( i=1,2\ldots, n-1)$.

Let $x\in \mathds{R}^{3}$ be a point, $M$ a subset of $\mathds{R}^{3}$, and $L$ an affine subspace of $\mathds{R}^{3}$. Affine subspaces are always supposed to have dimension  at least $1$.
Denote by $\pi_{L}(x)$ and $\pi_{L}(M) $ the orthogonal projections of $x$ and $M$ onto $L$.

 Let $M_{1},M_{2}\subset \mathds{R}^{3}$.
If   $\overline{M_{1}}$  and  $\overline{M_{2}}$ are orthogonal, we say that  $M_{1}$ is \emph{orthogonal} to $M_{2}$, using the notation  $M_{1}\perp M_{2}$. If $\overline{M_{1}}$  and  $\overline{M_{2}}$ are parallel, we say that  $M_{1}$ is \emph{parallel} to $M_{2}$, and write  $M_{1}\parallel  M_{2}$.

For $ x_{1},x_{2},\ldots, x_{n}\in \mathds{R}^3$, the set $x_{1}x_{2}\ldots x_{n} \subset \mathds{R}^3$ is called a \emph{polytope}; if $n$ is  minimal,
then $x_{1},x_{2},\ldots, x_{n}$ are its \emph{vertices}.
Like for every convex body,
each boundary point $x$ of a polytope $P$
belongs to some supporting hyperplane $H_x$.
If $\mathrm{dim}(P\cap H_x)=2$, $P\cap H_x$ is called a \emph{face}. If $\mathrm{dim}(P\cap H_x)=1 $, it is called an \emph{edge}. If every vertex belongs to exactly $3$ faces, then $P$ is said to be \emph{simple}.

Now, every face and every  edge of the polytope $P\subset \mathds{R}^3$ gets a mark $x~($or $y, $ or $z)$, if it is parallel to the $x$-axis (or $y$-axis,~or $z$-axis, respectively).
So, every edge has a set of 0 or 1 marks
and every face has a set of 0 or 1 or 2 marks.
Let $m(F)$ be the set of marks of the face or edge $F$.

\section{The orthogonal connectedness of polyhedral surfaces}\label{Sect-ocp}

Let $P\subset \mathds{R}^3$ be a polytope,
and $\mathcal{G}_P$ be the graph the vertices of which are the faces of $P$, two vertices of $\mathcal{G}_P$ being adjacent, if the corresponding faces have a common edge whose set of marks is different from those of the two faces.

If every face of $P$ gets a non-empty set of marks, and at least one of these sets has cardinality 2, we say that $P$ is $rich$.

\begin{thm}\label{Pifforconn}
If the boundary of the simple  polytope  $P\subset \mathds{R}^{3}$ is orthogonally
connected, then $P$ is rich and $\mathcal{G}_{P}$ is connected.
\end{thm}
\begin{proof}
Let $ F, F'$ be  faces of $P$.
Since $\mathrm{bd}P$ is orthogonally connected,  any face of $P$ has a non-empty mark set.
For any two points in  $\mathrm{bd}P$, one in $F$ and the other in $F'$, there exists
an orthogonal path $L$ joining them in $\mathrm{bd}P$. Let $\{F_1,\ldots, F_k\}$ be the
sequence of  faces  met by $L$ one by one, where $F_1=F,$ and $ F_k=F'$.
 Since $P$ is simple, $F_i,F_{i+1}$ have a common edge $E_{i}$ $(i=1,\ldots, k-1)$.
 If both $m(F_i)$ and $m(F_{i+1})$ have one mark each, then $m(E_{i})$  is empty,  which
 implies that $F_i,F_{i+1} $  are adjacent in $\mathcal{G}_{P} $.
 If both $m(F_i)$ and $m(F_{i+1})$ have mark sets with cardinality $2$, then  $m(E_{i})$ is
 different from them  and $F_i,F_{i+1} $  are adjacent in $\mathcal{G}_{P} $.

 Now, suppose $\mathrm{card}(m(F_i))=1, \mathrm{card}(m(F_{i+1}))=2 $.
Then,  $m(F_i)\not\subset m(F_{i+1})$ and  $m(E_{i})=\emptyset$.  Hence, $F_i,F_{i+1} $
are adjacent in $\mathcal{G}_{P} $.

Consequently, $ F_1,\ldots, F_k$ determine a path in $\mathcal{G}_{P} $ from  $F_1$ to
$F_k$,  and $\mathcal{G}_{P} $ is connected.

Suppose now that $P$ is not rich. Take arbitrarily a point $x$ on a face $F$ of $P$ and
start an orthogonal path there. It reaches an edge of $F$ at $x_1$, say. If $x_1$ is not a
vertex of $P$, the path continues on a face adjacent to $F$ until it reaches again an
edge.   If $x_1$  is a vertex of $P$, there are two possible continuations, until next
edges are met.
So, for the continuation of the path, one or two line-segments are added to the initial
line-segment in $F$. The path continues this way using  countably many line-segments.
It will stop only if it comes back to $x$.
In order for $\mathrm{bd}P$ to be orthogonally connected, the path must cover
$\mathrm{bd}P$. However, the path has (2-dimensional) measure $0$, and a contradiction is
obtained.
\end{proof}

\begin{figure}[htbp]
\centering
\includegraphics[width=1.5in]{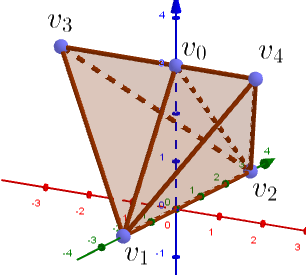}\ \ \ 
\caption{}
\label{santu}
\end{figure}

In the statement of Theorem \ref{Pifforconn} the orthogonal connectedness is essential.
Indeed, consider the  tetrahedron $P$  positioned  as in Figure \ref{santu}. It is simple and $\mathcal{G}_{P}$ is connected.  However, $P$ is not rich.  The boundary of $P$ is not orthogonally connected.

\begin{figure}[htbp]
\begin{center}
  \includegraphics[width=2in]{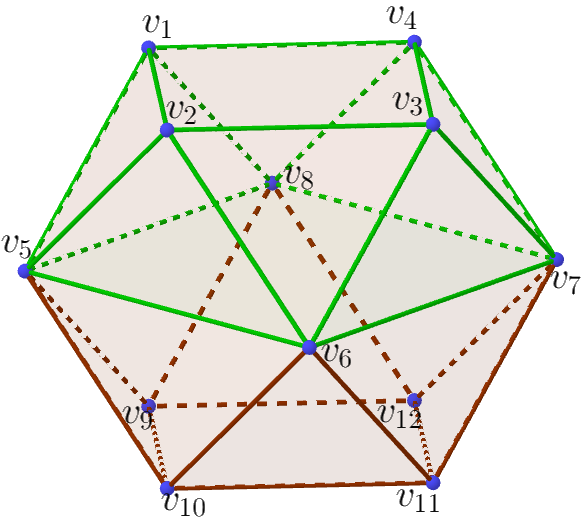}\ \
  \includegraphics[width=4in]{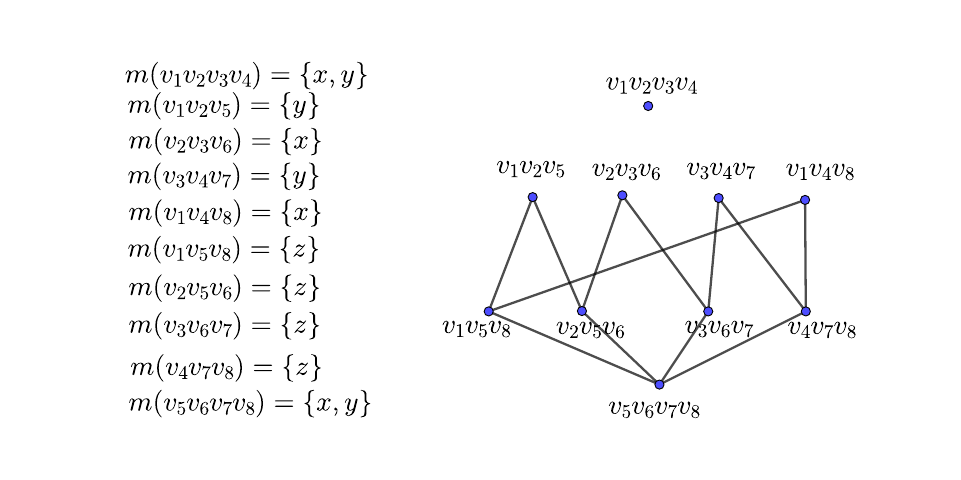}
  \end{center}
  \hspace{3cm} $(a)$ \hspace{8cm} $(b)$
   \caption{}
  \label{Gp1}
\end{figure}
The condition that $P$ be simple is also essential. Indeed,
 the half $P=v_1v_2v_3v_4v_5v_6v_7v_8$ of the cuboctahedron in Figure \ref{Gp1} $(a)$
 is rich and  $\mathrm{bd}P$ is orthogonally connected.  But  $\mathcal{G}_{P}$ is not connected (see Figure \ref{Gp1} $(b)$).  The polytope  $P$ is not simple.

\begin{figure}[htbp]
 \begin{center}
  \includegraphics[width=2in]{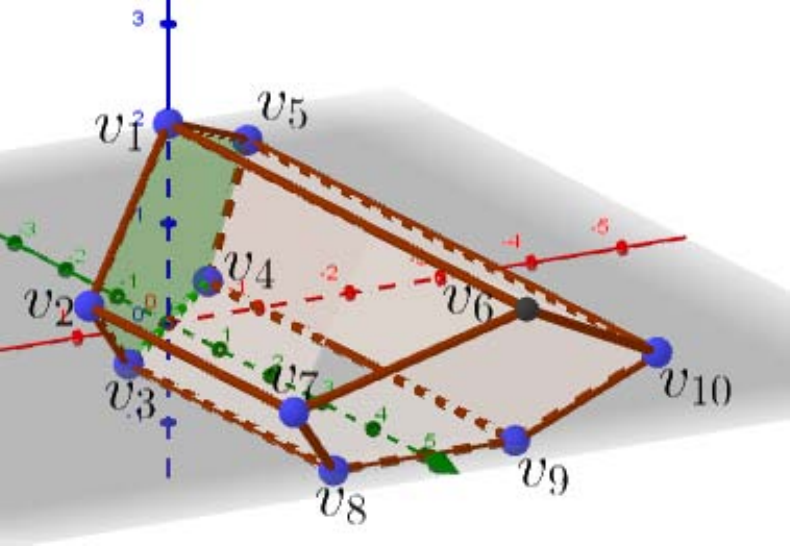} \includegraphics[width=4.3in]{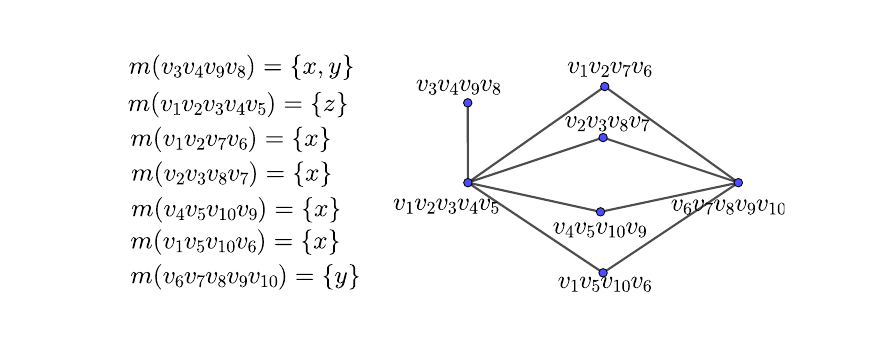}
  \end{center}
  \hspace{3cm} $(a)$ \hspace{7cm} $(b)$
   \caption{}
  \label{Gp2}
\end{figure}

Notice that the converse implication doesn't hold. Indeed, consider the points  $ v_1=(0,0,2),v_2= ( 2,2,1),v_3=(1,1,0),v_4=(-1,-1,0),v_5=(-2,-2,1),v_{6}=(0,7,2) ,v_7=(2, 6, 1), v_8=(1, 5, 0),v_9=(-1, 5, 0),v_{10}=(-2, 6, 1)$  (see Figure \ref{Gp2} $(a)$).
 The   simple  heptahedron  $P=v_1v_2v_3v_4v_5v_6v_7v_8v_9v_{10}$  is    rich and  $\mathcal{G}_{P}$ is  connected (see Figure \ref{Gp2} $(b)$). But   $\mathrm{bd}P$ is not orthogonally connected, because the orthogonal paths starting in $v_2$ only reach the points of $[v_2,v_7,v_{10}, v_5]$.

\section{Orthogonal decomposability of Platonic solids}\label{Sect-odps}

A polytope in $ \mathds{R}^3$ will be called \emph{orthogonally decomposable}, if it can
be decomposed into finitely many polytopes with orthogonally connected boundaries.

The boundary of the cube is obviously orthogonally connected. The following results show that two further Platonic  solids, which lack this property, are orthogonally decomposable.

\begin{thm}\label{regoctahedron}
The regular octahedron in its standard position can be decomposed into four tetrahedra with orthogonally connected
boundaries.
\end{thm}

\begin{proof}
Let $ P=v_1v_2v_3v_4v_5v_6\subset \mathds{R}^{3}$  be the regular octahedron, such that $
v_1v_5$ is parallel to the $x$-axis and $v_1v_6$ parallel to the $y$-axis (see Figure \ref{z8}).
 $P_1=v_1v_5v_6v_2, P_2=v_3v_5v_6v_2, P_3=v_1v_5v_6v_4, P_4=v_3v_5v_6v_4$ are four
 symmetric  tetrahedra.

Every point of $\mathrm{bd}P_1$ can be joined in  $\mathrm{bd}P_1$ with a point of $
v_1v_5v_6$ through an orthogonal path. For example, starting in  $x\in v_1v_2v_5$  we have the path $[x,x',x'']$ with $x'\in v_2v_5,  x''\in
v_5v_6, xx'\parallel v_1v_5,$ and $x'x''\parallel v_2v_4$.  Similarly,  from $y\in \mathrm{bd}P_1$,  we can reach a point $y''\in
 v_5v_6$. Now, the set of marks of $v_1v_5v_6$ is $\{x,y\}$. Then, there is an orthogonal
 path in $v_1v_5v_6 $ from $x''$ to $y''$, and thus an orthogonal path from $x$ to $y$ is completed.
 Symmetrically,  all $\mathrm{bd}P_i$ ($i=2,3,4$) are orthogonally connected.
\end{proof}

\begin{figure}[htbp]
\parbox[b]{0.35\textwidth}
  {\centering
  \subfigure
  {\includegraphics[width=2in]{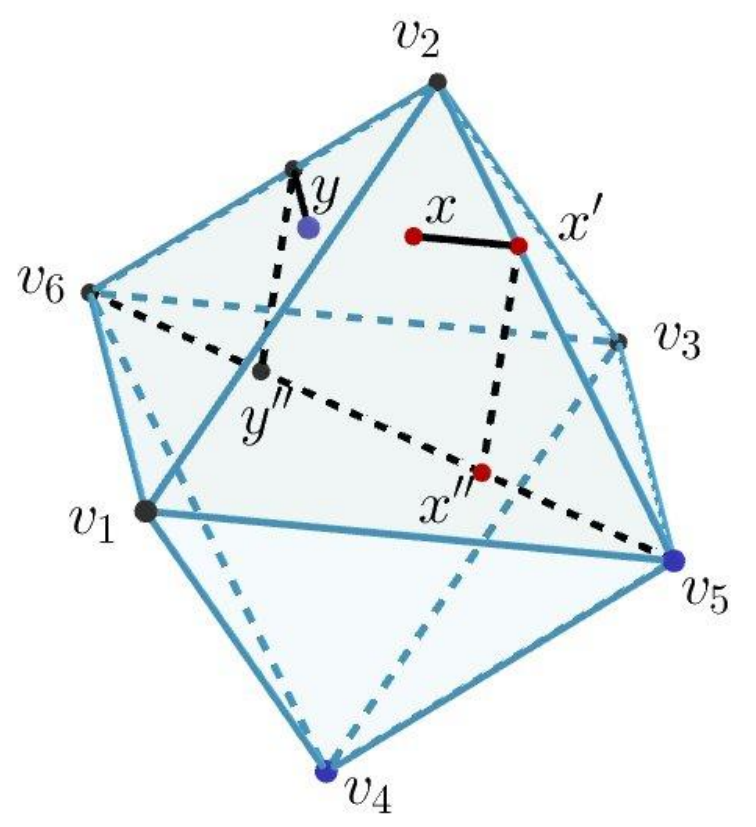}}\\
  \caption{ }
  \label{z8}}
  \parbox[b]{0.65\textwidth}
   {\centering
   \subfigure
  {\includegraphics[width=2in]{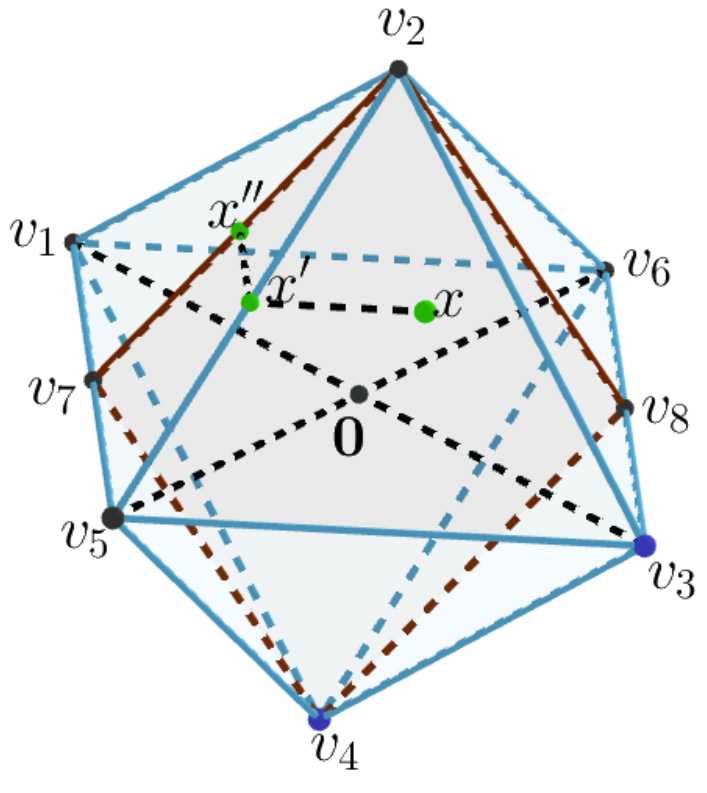}}\\
   \caption{}
  \label{figre8}}
\end{figure}

This is not the only way to suitably decompose the regular octahedron. We can decrease the
number of pieces to two.

\begin{thm}\label{thmreoc}
The regular octahedron in its standard position can be decomposed into two heptahedra with orthogonally connected boundaries.
\end{thm}
\begin{proof}

Let $v_1v_2v_3v_4v_5v_6$ be the regular octahedron, with $\mathbf{0}$ as the centre of $v_1v_6v_3v_5$ and $v_1=(1,-1,0)$;
see Figure~\ref{figre8}.

Denote $v_2v_7v_4v_8$ by $F_0$, $v_2v_5v_7$ by $F_1$, $v_4v_5v_7$ by $F_2$, $v_2v_3v_5$ by $F_3$, $v_3v_4v_5$ by $F_4$, $v_2v_3v_8$ by $F_5$ and $v_3v_4v_8$ by $F_6$. Then,
$m(F_0)=\left\{x,z\right\}$, $m(F_1)=m(F_2)=\left\{y\right\}$, $m(F_3)=m(F_4)=\left\{x\right\}$ and $m(F_5)=m(F_6)=\left\{y\right\}$.

$P_1=v_3v_5v_2v_4v_7v_8$ and $P_2=v_1v_6v_2v_4v_7v_8$ are two heptahedra.

Each point  $x\in \mathrm{bd}P_1$ can be joined in $\mathrm{bd}P_1$ with a point  $y\in F_0$ through an orthogonal path. For example, from $x\in F_3$ to $x''\in v_2v_7$ we have the path $[x,x',x'']$ with $x'\in v_2v_5$, $xx'\parallel v_3v_5$, $x'x''\parallel v_5v_7$.

Since $m(F_0)=\left\{x,z\right\}$, there is an orthogonal path in $F_0$ from $x''$ to $y$, and so an orthogonal path from $x$ to $y$ is completed.

Symmetrically, both $\mathrm{bd}P_1$ and $\mathrm{bd}P_2$ are orthogonally connected.
\end{proof}

Two disjoint edges in a polyhedron are \emph{facing each
 other}, if the
smallest distance between points on their lines
is realized by points in those edges.

\begin{thm}\label{tetrah}
 After a suitable rotation, any  tetrahedron  with two perpendicular edges facing each
 other  can be decomposed into two tetrahedra with orthogonally connected boundaries.
\end{thm}
\begin{proof}

Let $P=v_1v_2v_3v_4\subset \mathds{R}^{3}$ be a tetrahedron with $v_1v_2\perp v_3v_4 $,
such that $v_1v_2$ lies in the $y$-axis, ${\bf 0}\in v_1v_2$,
$v_3v_4 $ is parallel to the $x$-axis and
the intersection  of $v_3v_4$ with the
$z$-axis is not empty. Let $v_0\in v_3v_4$ belong to the $z$-axis. See Figure \ref{santu}.
Put $ P_1=v_1v_2v_3v_0$ and $ P_2=v_1v_2v_4v_0$; then  $ P=P_1\cup P_2$.

Every point of $\mathrm{bd}P_1$ can be joined in  $\mathrm{bd}P_1$ with a point of $
v_1v_2v_0$ through an orthogonal path,  each edge of which  is parallel to $
v_1v_2$ or $v_3v_0$.
Since $ m(v_1v_2v_0) =\{y,z\}$,  any two points in $v_1v_2v_0$ are joined by  an
orthogonal path inside $ v_1v_2v_0 $.  Thus, $\mathrm{bd}P_1$ is orthogonally
connected.
For $ \mathrm{bd}P_2$, the proof is analogous.
\end{proof}
\begin{cor}\label{regtetrah}
After a suitable rotation, the regular  tetrahedron   can be decomposed into two
tetrahedra with orthogonally connected boundaries.
\end{cor}

In the rest of the paper it is - often tacitly - understood that a suitable rotation is performed.

\section{Orthogonal decomposability of  Archimedean solids}
\label{Sect-odas}

In this section we find that four well-known  Archimedean solids are orthogonally decomposable.
The following lemma from \cite{orthhe} is useful.

\begin{lem}\label{orcylin}
If $M \subset \mathds{R}^{2}$ is orthogonally connected, then  any right cylinder in $\mathds{R}^3$ based on $M$,
as well as its  boundary, are orthogonally connected.
\end{lem}

We first show that the cuboctahedron is orthogonally decomposable.

\begin{thm}\label{cuboctahedron}
The  cuboctahedron  can  be decomposed into ten pentahedra with orthogonally connected
boundaries.
\end{thm}
\begin{proof}

\begin{figure}[h]
\centering
\includegraphics[width=2.4in]{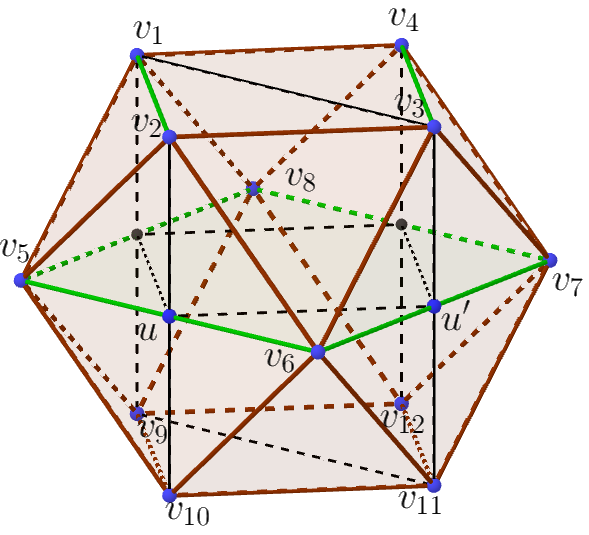}
\caption{ }
\label{jieban}
\end{figure}

Let $P\subset \mathds{R}^{3}$ be a cuboctahedron with vertices $v_i$ ($i=1,2,\ldots,12$)
such that $v_1v_2 $ is parallel to the $y$-axis and $v_1v_4 $ is parallel to the $x$-axis;
see Figure \ref{jieban}.

$Z_1=v_1v_2v_3v_9v_{10}v_{11}$ and $Z_2=v_1v_3v_4v_9v_{11}v_{12} $ are two right prisms
with bases $v_9v_{10}v_{11}$, $v_9v_{11}v_{12}$. Since $v_9v_{10}v_{11}, v_9v_{11}v_{12}$
are orthogonally connected,  both $\mathrm{bd}Z_1,\mathrm{bd}Z_2 $ are orthogonally
connected, since the boundary of any right cylinder is orthogonally connected, by Lemma \ref{orcylin}.

Notice that $\mathrm{cl}(P\setminus (Z_1\cup Z_2)) $ is the union four symmetric
pentahedra. Now, we consider one of them, say  $ K=v_2v_3v_6v_{10}v_{11} $. Take a plane $
H$ through $v_6$  orthogonal to the $ z$-axis, and put $\{u\}=H\cap v_2v_{10}, \{
u'\}=H\cap v_3v_{11}$. Put  $K'= uv_6u'v_{10}v_{11}$  and $K''= uv_6u'v_2v_3$, so $K=K'\cup
K''$.

Every point of $\mathrm{bd}K'$ can be joined in  $\mathrm{bd}K'$ with some point of $
F=uv_6u'$ through an orthogonal path.
Choose a point $p\in \mathrm{bd}K'$.
 If $p\in v_6v_{10}v_{11} $, then $[p,p',p'']$ is an orthogonal path from $p$ to $p''\in
 F$ in $\mathrm{bd}K' $, where  $pp'\parallel v_{10}v_{11} $, $p'\in  v_6v_{10}$  and
 $p''=\pi_{\overline{v_6u}}(p')$.
 If $p\in \mathrm{bd}K' \setminus v_6v_{10}v_{11}$, then $p\pi_{\overline{F}}(p)$ is
 parallel to the $z$-axis.

Since $ F$ has as mark set $\{x,y\}$, for any two points in it, there is an orthogonal
path connecting them in $F$.
Thus,  the boundary of $K'$ is orthogonally connected. For the boundary of  $ K'' $, the
proof is analogous. Hence,
$K$ can be decomposed into two pentahedra with orthogonally connected boundaries. The
other three  pentahedra in $\mathrm{cl}(P\setminus (Z_1\cup Z_2)) $  can be similarly decomposed.
\end{proof}

This is not the only way to suitably decompose the cuboctahedron. We can decrease the
number of pieces to two.
\begin{thm}
The  cuboctahedron  can  be decomposed into two  decahedra with orthogonally connected
boundaries.
\end{thm}
\begin{proof}
\begin{figure}[h]
\centering
\includegraphics[width=2.4in]{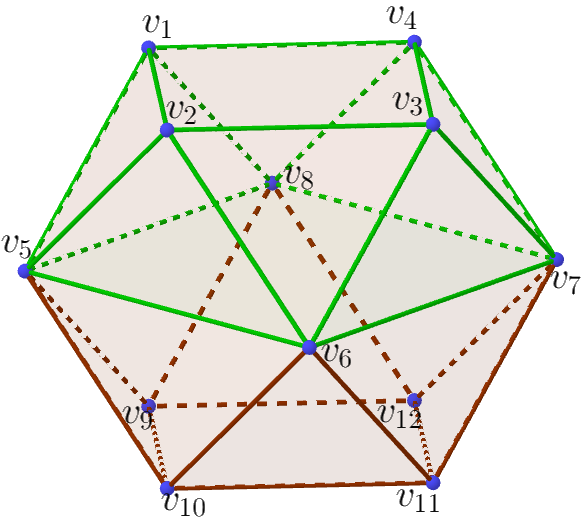}
\caption{ }
\label{jieban2}
\end{figure}

 Let $P\subset \mathds{R}^{3}$ be a cuboctahedron with vertices $v_i$ ($i=1,2,\ldots,12$)
 such that $v_1v_2 $ is parallel to the $y$-axis and $v_1v_4 $ is parallel to the
 $x$-axis; see Figure \ref{jieban2}.
 Put $P_1=v_1v_2v_3v_4v_5v_6v_7v_8$ and $ P_2=v_5v_6v_7v_8v_9v_{10}v_{11}v_{12}$; then  $
 P=P_1\cup P_2$.

 Every point of $\mathrm{bd}P_1$ can be joined in  $\mathrm{bd}P_1$ with a point of $ F=
 v_5v_6v_7v_8$ through an orthogonal path. Indeed, choose $p\in \mathrm{bd}P_1$.
 If $p\in v_1v_2v_3v_4$, then $[p,\pi_{\overline{v_2v_3}}(p), v_2,
 \pi_{\overline{v_5v_6}}(v_2)]$ is  an orthogonal path in $ \mathrm{bd}P_1$ joining $p$
 and $  \pi_{\overline{v_5v_6}}(v_2) \in F$.
 If $p\in v_2v_3v_6$, take $p'\in v_2v_6$ such that $pp'\parallel v_2v_3$.
The  orthogonal path  $[p,p',\pi_{\overline{v_5v_6}}(p')]$ joins $p$ and
 $\pi_{\overline{v_5v_6}}(p')\in   F$.   Similarly, if $p$ is in  another triangular face  of $P_1$
 with mark set $\{x\}$ or $\{y\}$, then it can be connected to a point of $F$ with an orthogonal path.
If $p$ is in $v_1v_5v_8$ or $v_2v_5v_6 $ or $ v_3v_6v_7$ or $v_4v_7v_8$, $[p,\pi_{\overline{F}}(p)]$ alone is a  suitable orthogonal
path.

 Since $F$ has as mark set $\{x,y\}$, for any two points in it there is an orthogonal
 path connecting them in $ F$.
 Thus, $\mathrm{bd}P_1$ is orthogonally connected.
For $ \mathrm{bd}P_2$, the proof  is analogous.
\end{proof}

The following lemma appears in \cite{DYZ}.

\begin{lem}\label{co-oc}
Any connected open set is orthogonally connected.
\end{lem}

\begin{lem}\label{cp-oc}
If $M \in \mathds{R}^{2}$ is a polygon every acute angle $\angle uvw$
of which admits a horizontal or vertical line-segment $vv'\subset M$,
then $M$ is orthogonally connected.
\end{lem}

\begin{proof}
By the hypothesis, for any $x, y\in M$,
there exist points $x', y'\in M$ not vertices of $M$, such that
$xx', yy'\subset M$ are horizontal or vertical.
Further, there exist points $x'', y''\in \mathrm{int}M$
such that $x'x'', y'y''\subset M$ are horizontal or vertical.
Since $\mathrm{int}M$ is a connected open set,
$x''$ and $y''$ are joined by an orthogonal path in $\mathrm{int}M$,
by Lemma \ref{co-oc}.
Hence, $x$ and $y$ are connected by an orthogonal path in $M$.
\end{proof}

\begin{thm}\label{trunoctahedron}
The truncated octahedron  can  be decomposed into four octahedra with orthogonally
connected boundaries.
\end{thm}
\begin{proof}
\begin{figure}[h]
\centering
\includegraphics[width=2.5in]{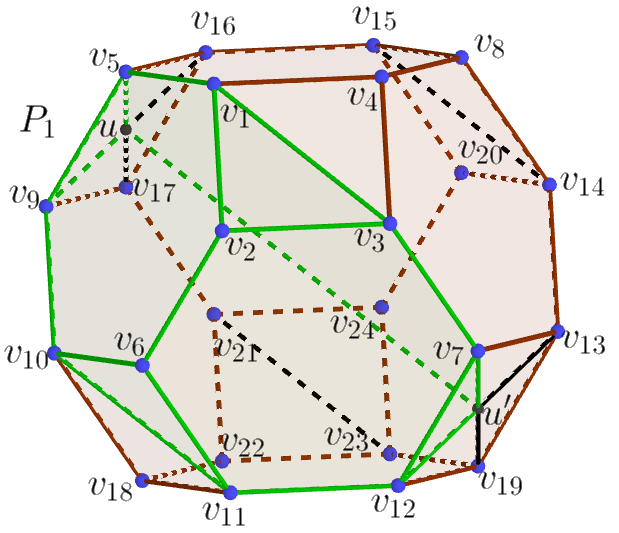}
\caption{ }
\label{jiejiaoba}
\end{figure}
Let $P\subset \mathds{R}^{3}$ be a truncated octahedron with vertices $v_i$
($i=1,2,\ldots,24$) such that $v_1v_2 $ is parallel to the $y$-axis and $v_1v_4 $ is
parallel to the $x$-axis; see Figure \ref{jiejiaoba}.
Let $ H\supset v_1v_3$ be the plane parallel to  the $z $-axis. By symmetry, $v_1v_5,
v_3v_7\subset H$. Put $\{u\}=H\cap  v_9v_{16}, \{u'\}=H\cap  v_{12}v_{13}$.
Let $P_1=v_1v_2v_3 v_7u'u v_5v_9v_{10}v_6v_{11}v_{12}$,
$P_2=v_4v_3v_1v_5 uu' v_7v_{13}v_{14}v_8v_{15}v_{16}$,
$P_3=v_{21}v_{22}v_{23}v_{19}u'u  v_{17}v_9v_{10}  v_{18}v_{11}v_{12}$,
$P_4=v_{24}v_{23}v_{21}v_{17} uu'  v_{19}$  $v_{13}v_{14}v_{20} v_{15}v_{16}$. Then $P=P_1\cup
P_2\cup P_3\cup P_4$.

Every point of $\mathrm{bd}P_1$ can be joined in  $\mathrm{bd}P_1$ with a point of $ F=
v_9v_{10}v_{11}v_{12}u'u$ through an orthogonal path. Choose a point $p\in
\mathrm{bd}P_1$.
If $p\in v_1v_2v_3v_4$, $ [p,\pi_{\overline{v_1v_2}}(p),v_1, \pi_{\overline{F}}(v_1)]$ is
an orthogonal path connecting $p$ and $\pi_{\overline{F}}(v_1)\in F$.
For $p\in v_1v_2v_6v_{10}v_9v_5 \cup v_2v_3v_7v_{12}v_{11}v_6$,  let $ H_p\ni p$ be a
plane parallel to $ v_1v_2v_3v_4$. If $H_p\cap v_2v_6\neq \emptyset$, then $p$ can be connected to some point $p'$ of $\mathrm{bd} v_1v_3v_5v_7uu'$ along $L=H_p\cap \mathrm{bd}P_1$. $[p,p',
\pi_{\overline{F}}(p')]$ is an orthogonal path connecting $p$ and
$\pi_{\overline{F}}(p')\in F$.
If $H_p\cap v_2v_6= \emptyset$,  then $H_p\cap v_6v_{10}\neq \emptyset$ and  $p$ can be connected to  a point  of
$\mathrm{bd}v_6v_{10}v_{11}$ along $L$.
Because all $v_6v_{10}v_{11}, v_5v_{9}u,v_7v_{12}u', v_1v_3v_5v_7uu' $  have as mark set
$\{z\}$, for any point $q$ in them, $q\pi_{\overline{F}}(q)$ is an orthogonal path joining
$q$ and $\pi_{\overline{F}}(q)\in F$.
Since all angles of $F$ are obtuse,
for any two points in $F$, there is an
orthogonal path in $F$, by Lemma \ref{cp-oc}.
 Then, $  \mathrm{bd}P_1 $ is orthogonally connected. Similarly,
all $\mathrm{bd}P_i$ $(i=2,3,4)$ are  orthogonally connected.
 \end{proof}

\begin{thm}\label{truncube}
The truncated cube  can  be decomposed into two decahedra with orthogonally connected
boundaries.
\end{thm}
\begin{proof}
\begin{figure}[h]
\centering
\includegraphics[width=2.2in]{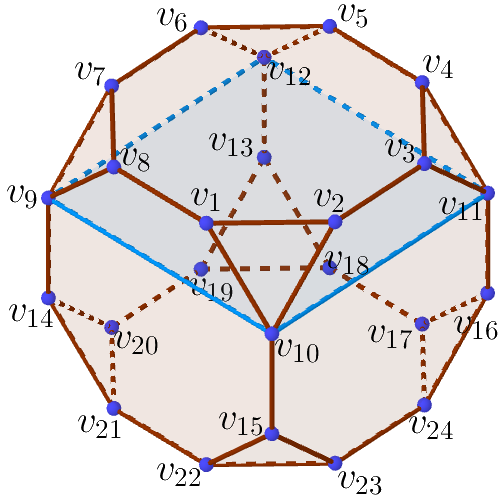}
\caption{ }
\label{jiejiaolifa}
\end{figure}
Let $P\subset \mathds{R}^{3}$ be a truncated cube with vertices $v_i$ ($i=1,2,\ldots,24$)
such that $v_1v_2 $ is parallel to the $x$-axis and $v_3v_4 $ is parallel to the $y$-axis;
see Figure \ref{jiejiaolifa}.  Then, the plane including  the square
$v_9v_{10}v_{11}v_{12} $  divides $P$ into two decahedra.
Let $v_i $ ($i=1,\ldots,12$) be the vertices of $P_1$ and $ v_j$ ($j=9,\ldots,24$)  the
vertices of $P_2$. Then $P=P_1\cup P_2$.  In $\mathrm{bd}P_1 $, triangles have as mark set
$ \{x\}$ or $ \{y\}$, the mark set of isosceles trapezoids is $ \{z\}$ and the square and the octagon
have as mark set $ \{x,y\}$.

Notice that  any point of $\mathrm{bd}P_1$ can be joined in  $\mathrm{bd}P_1$ with a point
of the square $ F= v_9v_{10}v_{11}v_{12}$ through an orthogonal path.
Choose a point  $a\in \mathrm{bd}P_1$. If $a$ is in a triangular face, then there is a
line-segment starting at $a$, parallel to the $x$-axis or the $y$-axis,  meeting an edge of an isosceles
trapezoid.
From these we continue orthogonally onto $F$.
If $a $ is in the octagon,
there is an orthogonal path connecting
 $a$ to some point  $b$ of a isosceles trapezoid, and we continue as before.
  Since $F$ is orthogonally
 connected, for any two points in $F$, there exists an orthogonal path joining them.
 Hence, $ \mathrm{bd}P_1$  is   orthogonally connected. Similarly,   $\mathrm{bd}P_2$ is
 also orthogonally connected.
 \end{proof}

\begin{thm}\label{a1}
 The  truncated tetrahedron can be decomposed into two  heptahedra with orthogonally connected boundaries.
\end{thm}
\begin{proof}\begin{figure}[h]
\centering
\includegraphics[width=2in]{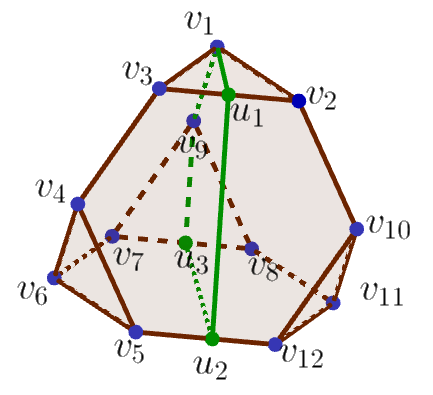}\\
\caption{}
\label{oda1}
\end{figure}
Let $P$ be a truncated tetrahedron  with vertices $v_i$ ($i=1,2,\ldots,12$)
such that $v_1v_9 $ is parallel to the $z$-axis and $v_2v_3 $ is parallel to the $x$-axis;
see Figure \ref{oda1}.
Let $H$ be the plane including $v_1v_9 $ orthogonal to $v_2v_3  $. Take $\{u_1\}=H\cap v_2v_3,\{u_2\}=H\cap v_5v_{12},\{u_3\}=H\cap v_7v_8$.  Then, $P_1=v_1 v_3v_4v_5v_6v_7v_9u_3u_2u_1, P_2=v_1v_2v_{10}v_{12}v_{11} v_8v_9 u_3u_2u_1 $ are two heptahedra and $P=P_1\cup P_2$.

By symmetry,  $ u_2$ is the midpoint of $v_5$ and $v_{12} $.
Then, the mark set of $F= v_1u_1u_2u_3v_9$ is $\{y,z\}$.
Moreover, the angle $\angle u_3u_2u_1$ is acute
and the other angles of $F$ are obtuse.
Let $w=(v_1+v_9)/2$.
Then $u_2w$ is parallel to the $y$-axis.
By Lemma \ref{cp-oc}, $F$ is orthogonally connected.
Every point $p\in \mathrm{bd}P_1$  can be joined in  $\mathrm{bd}P_1$ with a point
of  $ F$ through an orthogonal path.
 Indeed, if $p\in u_1u_2v_5v_4v_3\cup u_3u_2v_5v_6v_7 \cup v_7v_9v_3\cup u_1v_1v_3$
 then $[p,\pi_{\overline{F}}(p)]$ is a suitable orthogonal path. If $p\in v_4v_5v_6$, we reach $v_5v_6$ by travelling vertically. If $p\in  v_1v_3v_4v_6v_7v_9$, we move vertically until we meet $v_6v_7$ or $ v_7v_9$, and then continue as described above.
  Since $ F$ is orthogonally connected,
$ \mathrm{bd}P_1$  is   orthogonally connected, too.  Similarly,   $\mathrm{bd}P_2$ is
 also orthogonally connected.
\end{proof}

\section{Orthogonal non-decomposability}\label{Sect-ond}

In Theorem \ref{regoctahedron}, Corollary \ref{regtetrah}, Theorem \ref{cuboctahedron},
Theorem \ref{trunoctahedron},  Theorem \ref{truncube}  and Theorem \ref{a1},  we have seen that several
well-known Platonic and Archimedean solids are  orthogonally decomposable (after a suitable rotation).
Are all of them orthogonally decomposable?

If a polytope  is orthogonally decomposable,
then the mark set of each of its faces is non-empty.
Indeed, if  the mark set of one face $F$ is empty, then  there is no orthogonal path starting at any point in $\mathrm{int}F$ and the whole surface cannot be orthogonally decomposable.

\begin{thm}\label{dihe}
 Let $P \subset \mathds{R}^{3}$ be a polytope and $F$ one of its faces, different from a
 rectangle. If all the dihedral angles formed with $F$ are larger than $ 3\pi/4$, then $P$
 is not orthogonally decomposable.
\end{thm}
\begin{proof}
Suppose, on the contrary, that $P$ is orthogonally decomposable. Then,  every face must
have a non-empty mark set.

\underline{Case \uppercase\expandafter{\romannumeral1}. }  $m(F)=\{x,y\}$. Since $F$ is
not a rectangle, $F$ has an edge $E$ with $m(E)\not\in\{x,y\}$, which implies
$m(E)=\emptyset$. If $F'$ is the face meeting $F$ along $E$, necessarily $m(F')=\{z\}$.
But this means that $F\perp F' $, which is contradicting the hypothesis.

\underline{Case \uppercase\expandafter{\romannumeral2}. }    $m(F)=\{x\}$.
Let $Y$ and $Z$ be the $y$- and $z$-axis.
Among the angles $\angle(\overline{F},Y) $ and $\angle(\overline{F},Z) $, at most one is
less than $ \pi/4$, say $ \angle(\overline{F},Y)$. At most two parallel edges of $F$ have
as mark set $\{x\}$.

Let $F'$ be a  face  of $P$, neighbour of $F$, with $m(F\cap F')\neq \{x\}$. The angle
between any line in $\overline{F'}$ and $\overline{F}$ is less than $\pi/4 $, because
the dihedral  angle between $F$ and $F'$  is larger than $ 3\pi/4$.
Then, $m(F')= \{y\}$.

Thus, if $F$ is a triangle, it has two edges with mark sets different from $\{x\}$. It
follows that both faces meeting $F $ along these two edges have the same mark set $\{y\}$,
which is absurd.

If $F$ is not a triangle, it has at least two edges not having $\{x\}$ as mark set. Then,
again, all of them have mark set  $\{y\}$, which is impossible, $F$ being different from a
rectangle.
\end{proof}
\begin{cor}
The regular icosahedron,
the   rhombicuboctahedron, the icosidodecahedron,  the truncated dodecahedron, the
truncated icosahedron, the truncated icosidodecahedron, the  rhombicosidodecahedron, the
snub cube  and the snub  dodecahedron are not orthogonally decomposable.
\end{cor}
\begin{proof}
For the  regular icosahedron,  the dihedral angle between any two adjacent triangles
equals $\arccos (-\sqrt{5}/3)> 3\pi/4$; see \cite{dmcco}.
In the following eight cases, too, the dihedral angles are taken from \cite{dmcco}.

For the  rhombicuboctahedron, choose  $F$ to be a triangular face. Its dihedral angles
with the neighbouring squares equal $\arccos (-\sqrt{2/3})> 3\pi/4$.

For the icosidodecahedron and  the truncated dodecahedron,
consider a triangular face $F$.  Its dihedral angles with its neighbouring pentagons and  decagons equal \\
$\arccos
(-\sqrt{(5+2\sqrt{5})/15})> 3\pi/4$.

In the case of   the truncated icosahedron,  we  choose a pentagonal face  $F$. Its
dihedral angles with the  neighbouring hexagons equal \\$\arccos
(-\sqrt{(5+2\sqrt{5})/15})> 3\pi/4$, too.

For the truncated icosidodecahedron, consider a decagonal face  $F$.  Its dihedral angles
with the adjacent hexagons equal $\arccos (-\sqrt{(5+2\sqrt{5})/15})> 3\pi/4$ and with its
neighbouring squares  $\arccos (-\sqrt{(5+\sqrt{5})/10})> 3\pi/4$.

For the rhombicosidodecahedron, take  $F$ to be  a pentagonal face. Its dihedral angles with
the adjacent squares equal $\arccos (-\sqrt{(5+\sqrt{5})/10})> 3\pi/4$.

In the case of  the snub cube,  choose a triangular face $F$. Its dihedral angles with its
neighbouring squares  equal $ 142.98...^{\circ}>135^{\circ}$ and with the
neighbouring triangles  $ 153.23...^{\circ}> 135^{\circ}$.

For the snub  dodecahedron,  we consider a pentagonal face $F$. Its dihedral angles with the
neighbouring  triangles measure $152.93...^{\circ}> 135^{\circ}$.

Therefore, by Theorem \ref{dihe}, all these polytopes  are not orthogonally decomposable.
\end{proof}

\begin{thm}\label{z12}
The regular dodecahedron is not  orthogonally decomposable.
\end{thm}
\begin{proof}
\begin{figure}[h]
\centering
\includegraphics[width=0.5\textwidth]{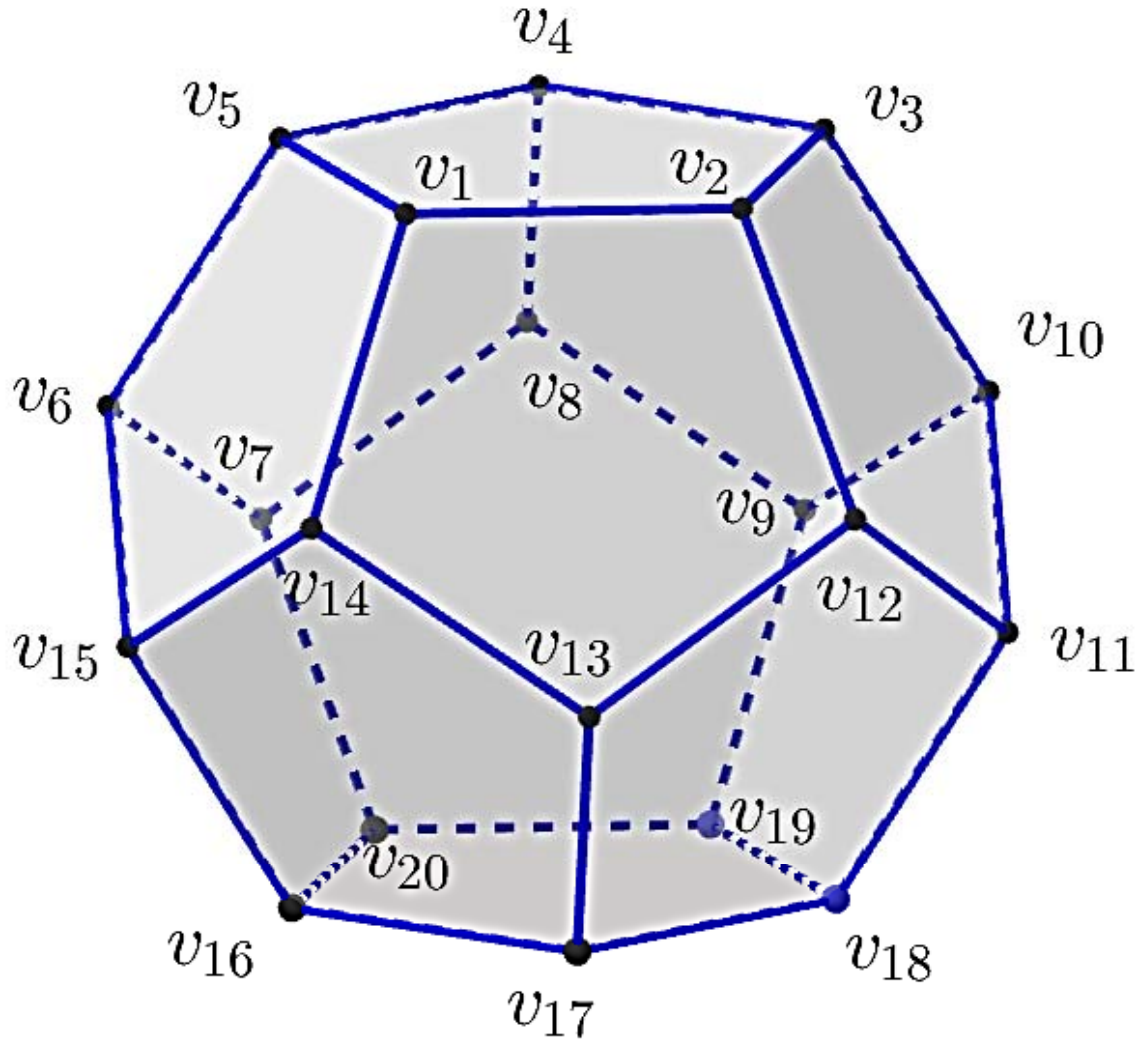}\\
\caption{}
\label{z12tu}
\end{figure}
Let $P$ be a regular dodecahedron.
 Suppose $P$ is  orthogonally decomposable.
 Take a pentagonal face $F$ of $P$. There are 5 pentagonal faces $F_i$ $(i=1,\ldots,5)$ adjacent to $F$.
 Each of those 6 faces has a mark (at least one). There are 3 marks.
 If  two faces have the same mark, say $x$, then their intersection line is parallel to the $x$-axis.
 So, if 3 faces have the same mark, then their intersection lines are parallel.
 Among the 15 intersection lines of our 6 faces, there are no parallel pairs, so no 3 faces have the same mark.
 Hence, the 6 faces form 3 pairs, each with the same mark. Suppose without loss of generality that $F,F_1$ form such a pair. Then the other two pairs are among $F_2,F_3,F_4,F_5$.
 If $F_2,F_3 $ have the same mark and $F_4,F_5$ too, then
 the lines $ \overline{F_2}\cap \overline{F_3}$  and  $ \overline{F_4}\cap \overline{F_5}$ should be orthogonal, but they are not. Neither are the lines $ \overline{F_2}\cap \overline{F_4}$  and $ \overline{F_3}\cap \overline{F_5}$, nor the lines $ \overline{F_2}\cap \overline{F_5}$  and $ \overline{F_3}\cap \overline{F_4}$. A contradiction is obtained.
\end{proof}

\begin{thm}\label{a10}
The truncated cuboctahedron is not  orthogonally decomposable.
\end{thm}
\begin{proof}
\begin{figure}[h]
\centering
\includegraphics[width=0.3\textwidth]{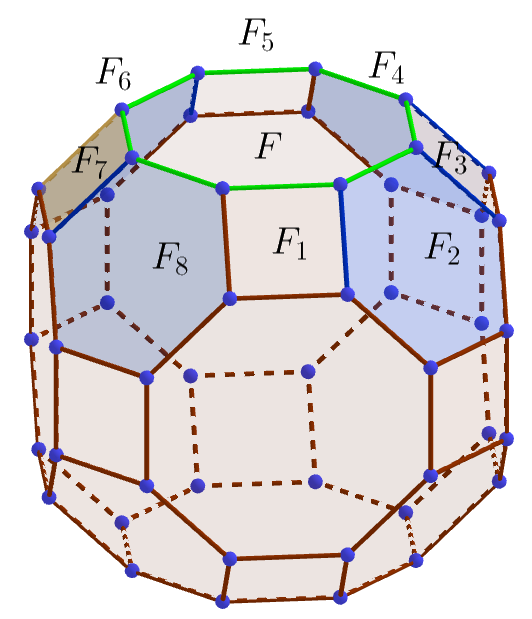}\\
\caption{}
\label{a10tu}
\end{figure}
Let $P$ be a truncated cuboctahedron.   Suppose $P$ is  orthogonally decomposable.
 Consider an octagonal face $F$ of $P$.  Then $ \mathrm{card}(m(F))\in \{1,2\}$. If $\mathrm{card}(m(F))=2 $, then $F$ has at most 4 edges parallel to some coordinate axis. If $E$ is one of the other edges of $F$, and $F'$ is the  face of $P$ with  $F\cap F'=E$, then $m(F')=\emptyset$, because $F$ and $F'$ are not perpendicular. This contradicts the  orthogonal decomposability of $P$.

 If $\mathrm{card}(m(F))=1 $, say $m(F)=\{x \} $, then at most 2 edges of $F$ have non-empty mark set.
   So, there are at least 6 faces adjacent to $F$ with at least  one mark each, different from  $\{x\}$.   Hence,  at least 3 faces have the same mark, say $\{y\} $.  This means that the three intersection lines of their planes are parallel to the $y$-axis. Among the triples of faces adjacent to $F$, only faces positioned like $F_8,F_1,F_2$ (see Figure \ref{a10tu}) have this property. Suppose without loss of generality they are precisely $F_8,F_1,F_2$. Then $m(F_1\cap F_2)=\{y\}$,  which implies that $m(F\cap F_1)=\{x\}$. But then, $m(F_3)=\emptyset$,
   which  contradicts the hypothesis.
\end{proof}

\section{Epilogue}

As a conclusion, the boundary of the cube is orthogonally connected,
and 2 Platonic solids and 4 Archimedean ones are
orthogonally decomposable, while the remaining 2 Platonic
and 9 Archimedean solids are not.

It would be interesting to know which polytopes among
the Catalan and Johnson solids have orthogonally
connected boundaries or are orthogonally decomposable.

Another question is: do there exist interesting classes of
non-convex polyhedral surfaces, which are orthogonally connected?

Moreover, one can investigate non-polyhedral topological spheres which
are orthogonally connected.

\section*{Acknowledgements}

This work was supported by NSF of China (12271139, 12201177, 11871192); the High-end Foreign Experts Recruitment Program of People's
Republic of China (G2023003003L); the Program for Foreign Experts of Hebei Province; the Hebei Natural Science
Foundation (A2023205045); the Special Project on Science and Technology Research and Development Platforms, Hebei Province (22567610H);
the Science and Technology
Project of Hebei Education Department (BJK2023092)
and the China Scholarship Council.

\end{document}